\documentclass{article}
\usepackage{indentfirst}
\usepackage[tbtags]{mathtools}
\usepackage{amssymb}
\usepackage{amsmath}
\usepackage{amsthm,verbatim}
\usepackage{amsrefs}
\usepackage[top=1 truein,bottom=1.5 truein,hmargin={1 truein,1 truein}]{geometry}
\usepackage[all,cmtip]{xy}
\mathtoolsset{showonlyrefs,showmanualtags}

\usetagform{default}
\makeatletter
\newtheorem*{rep@theorem}{\rep@title}
\newcommand{\newreptheorem}[2]{%
\newenvironment{rep#1}[1]{%
 \def\rep@title{#2 \ref{##1}}%
 \begin{rep@theorem}}%
 {\end{rep@theorem}}}
\makeatother

\theoremstyle{plain}
\newtheorem{theorem}{Theorem}
\theoremstyle{definition}
\newtheorem*{definition}{Definition}

\newtheorem{lemma}{Lemma}[section]
\newtheorem{proposition}{Proposition}[section]

\newtheorem{conjecture}{Conjecture}[section]
\newreptheorem{theorem}{Theorem}
\newreptheorem{proposition}{Proposition}

\newcommand{\Qd}[1]{\mathbb{Q}(\sqrt{#1})}

\newcommand{\Z}{\mathbb{Z}}
\newcommand{\Q}{\mathbb{Q}}

\newcommand{\A}{\alpha}

\newcommand{\e}{\equiv}

\newcommand{\OK}{\mathcal{O}_{K}}

\newcommand{\lcm}{\operatorname{lcm}}
\newcommand{\ord}{\operatorname{ord}}

\DeclareMathOperator{\rad}{rad}

\title{Wall's Conjecture and the ABC Conjecture}
\author{George Grell, Wayne Peng}
\date{\today}

\begin{document}

\maketitle	
\begin{abstract}
%We investigate a connection between the $abc$ conjecture of Masser-Oesterl\'{e}-Szpiro for number fields and the number of occurrences of Fibonacci-Wieferich primes. Assuming this conjecture we show that there are infinitely many non-Fibonacci-Wieferich primes. We also provide a new heuristic for the number of such primes beneath a certain value.
We show that the $abc$ conjecture of Masser-Oesterl\'{e}-Szpiro for number fields implies that there are infinitely many non-Fibonacci-Wieferich primes. We also provide a new heuristic for the number of such primes beneath a certain value.
\end{abstract}
\section{Wall's conjecture}
The Fibonacci sequence $\{F_n\}$, our main character in this story, is a sequence generated by the relation
\[
F_n=F_{n-1}+F_{n-2}\qquad\forall n\geq 2
\]
with initial values $F_0=0$ and $F_1=1$. What happens if we consider the Fibonacci sequence modulo an integer? Let us compute $F_n\mod 3$. It is
\[
0,\ 1,\ 1,\ 2,\ 0,\ 2,\ 2,\ 1,\ 0,\ 1,\ldots.
\]
What about modulo $4$ or modulo $5$? What about modulo any integer? It is not hard to observe that the sequences always form a cycle. We are curious about the length of these cycles. We use $\pi(m)$ to represent the period of the Fibonacci sequence modulo $m$. In 1960, D.D. Wall\cite{Wall1960} investigated this function and gave the following beautiful theorems.
\begin{theorem}[Wall]
The function $\pi$ is a multiplicative function, i.e. if $m$ is a positive integer, and $m=p_1^{e_1}p_2^{e_2}\cdots p_r^{e_r}$ where $p_i$ are distinct primes, $e_i$ and $r$ are positive integers, then
\[
\pi(m)=\pi(p_1^{e_1})\pi(p_2^{e_2})\cdots \pi(p_r^{e_r}).
\]
\end{theorem}

In order to understand a multiplicative function, it is sufficient to study the value of the function at prime powers. This motivation gives us the following theorem.
\begin{theorem}[Wall]\label{thm:wall}
For all prime $p$, we have the following two cases:
\begin{enumerate}
\item If $p\e \pm 1\mod 10$, then $\pi(p)$ divides $p-1$.
\item If $p\e \pm 3\mod 10$, then $\pi(p)$ divides $2(p+1)$.
\end{enumerate}
Moreover, if we have $\pi(p^e)\neq \pi(p^{e+1})$, then
\[
\pi(p^{e+i})=p^i\pi(p^e)\qquad\forall i\geq 1.
\]
\end{theorem}
This seems to be the end of the whole story, but like all good stories, it never ends so easily. At the end of his paper, Wall offered the conjecture that $\pi(p)\neq\pi(p^2)$ for all primes $p$. This is called Wall's conjecture, and has remained unsolved for over 50 years. A small step to answer this question is to show that infinitely many primes satisfy Wall's conjecture, and that is the goal of this paper.

\section{Preliminary}
We are curious about the position where zero first occurs in the sequence $\{F_n\mod m\}$. Let $l:=l(m)$ be the smallest integer such that $F_{l}\e 0\mod m$. The identity $F_{n+k}=F_nF_{k+1}+F_{n-1}F_{k}$ shows that
\[
F_{2l}\e F_{l}F_{l+1}+F_{l-1}F_l\e 0\mod m.
\]
Using induction, one can easily show the following statement:
\begin{lemma}[Wall]\label{lem:1}
If $F_n\e 0\mod m$, then $l(m)$ divides $n$.
\end{lemma}
This implies that $l(p)\leq l(p^2)$. Now, let us look closely at $F_{l(p^e)}$. It is well-known that the closed form of $F_n$ is
\[
F_n=\dfrac{\A^n-\bar{\A}^n}{\sqrt{5}}
\]
where $\A=(1+\sqrt{5})/2$ and $\bar{\A}=(1-\sqrt{5})/2$. Thus, the ring $\Z$ is slightly small for our discussion. Instead we will work in the larger ring $\Z[\A]$. The reason for this choice is that $\Z[\A]$ is the ring of integers of $\Qd{5}$. Thus, the modulus $p^e$ is considered in the ring $\Z[\A]/p^e\Z[\A]$. The disadvantage of this ring is that $\Z[\A]/p^e\Z[\A]$ has zero divisors, so some elements have no inverse. However, we should note that $\A\bar{\A}=-1$, so $\A$ and $\bar{\A}$ are invertible in this ring. Also, if we choose $p\neq 2$ or $5$, then $2$ and $\sqrt{5}$ are both invertible. The assumption $p\neq 2$ or $5$ is not a big deal since it can be checked that both 2 and 5 satisfy Wall's conjecture

Denote $l(p^e)$ by $l$. We will show that $l$ is the order of $(\A\bar{\A}^{-1})$ in the multiplicative group $(\Z[\A]/p^e\Z[\A])^{\times}$. Let $n_0$ be the order of $(\A\bar{\A}^{-1})$, so by definition
\[
(\A\bar{\A}^{-1})^{n_0}\e 1\mod p^e.
\]
Hence, this implies 
\[
\A^{n_0}-\bar{\A}^{n_0}\e 0\mod p^e,
\]
and it follows that $F_{n_0}\e 0 \mod p^e$. By Lemma~\ref{lem:1}, $l$ divides $n_0$. Conversely, the definition of $l$ shows that
\[
\dfrac{\A^l-\bar{\A}^l}{\sqrt{5}}\e 0\mod p^e.
\]
Since $\sqrt{5}^{-1}$ is not a zero divisor, it follows that $(\A\bar{\A}^{-1})^l\e 1\mod p^e$. By Lagrange's Theorem, $l$ must be divisible by the order $n_0$, which implies our desired result, $l=n_0$, or formally
\begin{lemma}\label{lem:2}
For $p\neq 2$ or $5$, the integer $l(p^e)$ is the order of $(\A\bar{\A}^{-1})$ in the multiplicative group $(\Z[\A]/p^e\Z[\A])^\times$.
\end{lemma}

These observations lead us to the following key lemma.
\begin{lemma}[Main Lemma]\label{mainlemma}
We assume $p\neq 2$ or $5$. If $p$ is a prime dividing $F_n$, and $p^2$ does not divide $F_n$, then $\pi(p)\neq\pi(p^2)$.
\end{lemma}
\begin{proof}
Denote $l(p)$ by $l$. By Lemma~\ref{lem:1} and Lemma~\ref{lem:2}, our assumption implies
\[
(\A\bar{\A}^{-1})^{l}\e 1+kp\mod p^2
\]
for some $k\not\e 0\mod p$. The order of $(1+kp)$ in the multiplicative group $\Z[\A]/p^2\Z[\A]$ is $p$, so $l(p^2)=pl$. Since $l(p)$ divides $\pi(p)$ and $\pi(p)$ divides $p^2-1$, which is relatively prime to $p$, $l(p^2)$ doesn't divide $\pi(p)$, which implies $\pi(p)\neq \pi(p^2)$  by Theorem~\ref{thm:wall}.
\end{proof}
The implication of the main lemma is that if we can show that the squarefree part of $F_n$,
\[
\prod_{p\| F_n}p,
\]
goes to infinity where $p\| N$ if and only if $p|F_n$ but $p^2\not | F_n$, we will have shown that there are infinitely many primes $p$ such that Wall's conjecture holds. Unfortunately, in this paper we are unable to show this without another important conjecture, the $abc$ conjecture.
\section{$abc$ conjecture}
In the following context, we only consider the number field $\Qd{5}$. The ring of integers of $\Qd{5}$ is $\Z[\A]=\{x+y\A\mid x,y\in\Z\}$ which is a generalization of integers $\Z$. One of the important properties $\Z[\A]$ shares with $\Z$, is that any element can be factored into a product of primes uniquely. In order words, for any $a\in\Z[\A]$ we can uniquely write
\[
a=uq_1^{e_1}q_2^{e_2}\cdots q_r^{e_r}
\]
where $q_i$ are prime elements of $\Z[\A]$ and $u$ is a unit of $\Z[\A]$. This fact is easy to show for $\Z[\A]$ via a Euclidean algorithm.

Recall that we can define prime elements and units in $\Z[\A]$ as follows.
\begin{definition}
Let $a,b$, and $q$ be elements of $\Z[\A]$. We call $q$ a prime element of $\Z[\A]$ if $q|ab$ (i.e. $ab=qc$ for some $c\in\Z[\A]$) implies
\[
q|a\qquad\text{or}\qquad q|b.
\]
\end{definition}
\begin{definition}
An element $u$ of $\Z[\A]$ is called a unit if we can find another element $v$ of $\Z[\A]$ satisfying
\[
uv=1.
\]
\end{definition}
A triple $(a,b,c)\in \Z[\A]^3$ is said to be written in the lowest term if the greatest common divisor $\gcd(a,b,c)$ of $a$, $b$ and $c$ is $1$. 

An element $a\in \Z[\A]$ can be written as $x+y\A$, and we define the conjugate of $a$ to be
\[
\bar{a}:=x+y\bar{\A}.
\]
The norm $N:\Z[\A]\to\Z$ is given by, for any $a=x+y\A\in\Z[\A]$,
\[
N(a):=a\bar{a}=(x+y\A)(x+y\bar{\A}).
\]
\begin{definition}
Let $(a,b,c)\in \Z[\A]^3$ be written in the lowest term. The height function is
\[
H(a,b,c)=\max\{|a|,|b|,|c|\}\max\{|\bar{a}|,|\bar{b}|,|\bar{c}|\}
\]
where $|\cdot|$ is the usual absolute value.
\end{definition}
Readers can refer to Bombieri and Gubler\cite{bombieri2007heights} for more detail about height functions. The height function is proportional to the algebraic complexity or number of bits needed to store a point. From this idea, we can expect that there are only finitely many points of which the complexity is bounded. The theorem for this finiteness property is called Northcott's theorem\cite{silverman_introduction_2006}. 

The idea of the $abc$ conjecture is that if the point $(a,b,c)$ is on the hyperplane $X+Y+Z=0$, i.e. $a+b+c=0$, and $(a,b,c)$ is written in the lowest term, then the complexity of this point can somehow be bounded by all prime factors of $a$, $b$ and $c$. We call the product of all prime factors of $a$, $b$ and $c$ the radical.
\begin{definition}
Let $(a,b,c)\in \Z[\A]^3$ be written in the lowest term. Let $q$ be a prime element of $\Z[\A]$. The radical of $(a,b,c)$, denoted by $\rad(a,b,c)$, is 
\[
\rad(a,b,c):=\prod_{q|abc}N(q).
\]
\end{definition}
With these definitions, we can state the $abc$ conjecture of Masser-Oesterl\'{e}-Szpiro.
\begin{conjecture}
Let $\varepsilon>0$, then there exists a constant $C_\varepsilon$ such that 
\[
H(a,b,c)\leq C_\varepsilon\rad_K(a,b,c)^{1+\varepsilon}
\]
for all $(a,b,c)\in\{(X,Y,Z)\in\Z[\A]^3\mid X+Y+Z=0\}$ with $(a,b,c)$ written in the lowest term.
\end{conjecture}
\section{Main theorem}
\begin{theorem}
There are infinitely many primes $p$ satisfying Wall's conjecture assuming the $abc$ conjecture.
\end{theorem}
\begin{proof}
Let $U_n:=\prod_{p\| F_n}p$ be the squarefree part of $F_n$ and rewrite $F_n=U_n V_n$. For the sake of contradiction, we suppose $U_n$ is bounded, i.e. there exists a constant $B$ such that $U_n<B$ for all $n$. 

We start our argument by considering the closed form of the Fibonacci numbers, and it follows that 
\[
\sqrt{5}F_n-\A^n+\bar{\A}^n=0.
\]
Since $\A$ is a unit, the point $(\sqrt{5}F_n,-\A^n,\bar{\A}^n)$ is written in the lowest term. Thus, the $abc$ conjecture implies that, for any $\varepsilon$, there exists a constant $C_\varepsilon$ such that
\[
H(\sqrt{5}F_n,-\A^n,\bar{\A}^n)\leq C_\varepsilon \rad(\sqrt{5}F_n,-\A^n,\bar{\A}^n)^{1+\varepsilon}.
\]
By definition, we have
\begin{align}
H(\sqrt{5}F_n,-\A^n,\bar{\A}^n)&=\max\{|\sqrt{5}F_n|,|-\A^n|,|\bar{\A}^n|\}\max\{|-\sqrt{5}F_n|,|-\bar{\A}^n|,|\A^n|\}\\
&\geq (\sqrt{5}F_n)^2=5F_n^2=5U_n^2V_n^2.
\end{align}
Now, we would like to estimate the radical. Since $\A$ and $\bar{\A}$ are units, all prime factors are coming from $\sqrt{5}F_n$. Therefore, we have
\[
\rad(\sqrt{5}F_n,-\A^n,\bar{\A}^n)=\prod_{p|\sqrt{5}F_n}N(p)\leq 5U_n^2V_n.
\]
It follows that
\[
5U_n^2V_n^2\leq C_\varepsilon(5U_n^2V_n)^{1+\varepsilon},
\]
and so
\[
V_n^2\leq C_{\varepsilon}5^{\varepsilon}U_n^\varepsilon V_n^{1+\varepsilon}.
\]
Since $U_n$ is bounded and $F_n$ tends to infinity, $V_n$ must also tend to infinity as $n$ tends to infinity. Thus, it is clearly a contradiction once we choose $\varepsilon<1$. By Lemma~\ref{mainlemma}, we conclude that there are infinitely many primes satisfying Wall's conjecture.
\end{proof}
Our argument can be easily generalized to a Fibonacci-like sequence. We merely need to change the height functions and radical into generalized forms. 

%Our story ends here, but the story for Wall's conjecture still has a long way to go. 

\section{A heuristic approach to Wall's conjecture}
Instead of providing a definitive statement as to the truth of Wall's Conjecture, we can give a heuristic that the conjecture is in fact false. There are two other heuristics that are in mild conflict, and our heuristic aims to reconcile the two. The first, from Crandall, Dilcher, and Pomerance, uses the fact that $p$ is a Fibonacci-Wieferich prime if and only if $F_{p-\left(\frac{p}{5}\right)}\equiv 0\mod{p^2}$ where $\left(\frac{p}{5}\right)$ is the Legendre symbol. Since it's always true that $F_{p-\left(\frac{p}{5}\right)}\equiv 0\mod{p}$, we can write $F_{p-\left(\frac{p}{5}\right)}\equiv kp\mod{p^2}$. Assuming $k$ is randomly distributed between 0 and $p-1$, there should be a $1/p$ chance that $p$ is a Fibonacci-Wieferich prime. Since the sum of $1/p$ over all primes diverges to infinity, this heuristic predicts infinitely many such primes\cite{CDP1997}. 

The second heuristic is due to Kla\v{s}ka and breaks the problem into two cases where either $p\equiv\pm1\mod5$ or $p\equiv\pm2\mod5$. He argues in the first case there is a $1/p^2$ chance of $p$ being a Fibonacci-Wieferich prime, and a $1/p$ chance in the second. While this still results in infinitely many Fibonacci-Wieferich primes, it predicts significantly less than the first heuristic\cite{Jiri2007}.

Our heuristic takes an approach similar to Kla\v{s}ka's by splitting into two cases, but by considering some extra information from the norm we argue that both should result in a $1/p$ chance. In our case, let $K=\Q(\A)$ and $\OK=\Z[\A]$ be the ring of integers of $K$ where $\A=(1+\sqrt5)/2$.
\begin{proposition}
A prime $p\neq2,5$ is a Fibonacci-Wieferich (or Wall-Sun-Sun) prime if and only if $\alpha^{p^2-1}\equiv1\mod{p^2}$ where the modulus is considered over $\OK$, i.e. the image of $\alpha^{p^2-1}$ under the canonical mapping $\varphi:\OK\to\OK/p^2\OK$ is the identity.
\end{proposition}
\begin{proof}
\iffalse
Fix a prime $p\neq2,5$ and let $e\in\Z^+$. Then $\pi(p^e)$ is smallest integer $\pi$ such that $F_\pi\equiv F_0\equiv0\mod{p^e}$ and $F_{\pi+1}\equiv F_1\equiv1\mod{p^e}$. Since $F_n=\sqrt{5}^{-1}(\alpha^n-\overline{\alpha}^n)$, this is equivalent to
\[
\begin{pmatrix}\sqrt5^{-1} & -\sqrt5^{-1}\\\sqrt5^{-1}\alpha & -\sqrt5^{-1}\overline{\alpha}\end{pmatrix}
\begin{pmatrix}\alpha^\pi\\\overline{\alpha}^\pi\end{pmatrix}\equiv
\begin{pmatrix}\sqrt5^{-1} & -\sqrt5^{-1}\\\sqrt5^{-1}\alpha & -\sqrt5^{-1}\overline{\alpha}\end{pmatrix}
\begin{pmatrix}1\\1\end{pmatrix}\mod{p^e}.
\]
Since the determinant of the $2\times2$ matrix is nonzero, we may cancel it from both sides to obtain
\[
\begin{pmatrix}\alpha^\pi\\\overline{\alpha}^\pi\end{pmatrix}\equiv\begin{pmatrix}1\\1\end{pmatrix}\mod{p^e}
\]
where we now take the modulus inside the ring $\OK$. Observe that this vector equation gives the characterization $\pi(p^e)=\lcm\{\ord_{p^e}(\A),\ord_{p^e}(\overline{\A})\}$ where $\ord_{p^e}(\A)$ denotes the order of $\A$ modulo $p^e$.
\fi
Recall a prime $p$ is defined to be a Fibonacci-Wieferich prime if $\pi(p)=\pi(p^2)$. It is known that $\pi(p^e)=\lcm\{\ord_{p^e}(\A),\ord_{p^e}(\overline{\A})\}$ where $\ord_{p^e}(\A)$ denotes the order of $\A$ modulo $p^e$\cite{KJ2008}. Now suppose $\pi(p)=\pi(p^2)$. Consider what happens to the ideal $(p)$ in $\OK$. As $p\neq2,5$, we have two cases. In the first case, $(p)$ factors as $(q_1)(q_2)$, so by the Chinese remainder theorem $(\OK/p\OK)^\times\cong (\OK/q_1\OK)^\times\times(\OK/q_2\OK)^\times$. Incidentally, this case is equivalent to the case that $p\equiv\pm1\mod5$. In particular, the order of $(\OK/p\OK)^\times$ is $(p-1)^2$, and we can see that any element must have order dividing $p-1$. Thus, $\ord_{p}(\A)$ divides $p-1$. In the second case, $(p)$ remains a prime ideal in $\OK$ in which case the order of $(\OK/p\OK)^\times$ is $p^2-1$. An immediate result is that $\ord_{p}(\A)$ divides $p^2-1$. In either case, $\ord_{p}(\A)$ divides $p^2-1$. The same argument also shows that $\ord_{p}(\overline{\A})$ divides $p^2-1$. Using the above characterization of $\pi(p^e)$ we conclude that $\pi(p)|p^2-1$. Thus $\pi(p^2)|p^2-1$ and so $\alpha^{p^2-1}\equiv1\mod{p^2}$.

Conversely, if $\alpha^{p^2-1}\equiv1\mod{p^2}$, then $\ord_{p^2}(\A)$ divides $p^2-1$. Since $p\nmid p^2-1$, it must be that $p\nmid\ord_{p^2}(\A)$. By our initial theorems from Wall, if $\pi(p^2)\neq\pi(p)$, then $\pi(p^2)=p\pi(p)$. If this is the case, then we see that $\A^{p\pi(p)}\equiv1\mod{p^2}$ and $\overline{\A}^{p\pi(p)}\equiv1\mod{p^2}$. Since either $\A^{\pi(p)}\not\equiv1\mod{p^2}$ or $\overline{\A}^{\pi(p)}\not\equiv1\mod{p^2}$ by assumption and $p$ is prime, it must be that $p$ divides $\ord_{p^2}(\A)$. This is a contradiction, so $\pi(p^2)=\pi(p)$.
\end{proof}
With our new definition, let $p$ be any prime not equal to 2 or 5. From an argument presented above, we know that $\A^{p^2-1}\equiv1\mod p$. So when considered modulo $p^2$, we see that $\A^{p^2-1}\equiv1+k_pp\mod{p^2}$ where $k_p=a+b\frac{1+\sqrt5}{2}$ and $0\le a,b\le p-1$. Although it may seem this implies a $1/p^2$ chance of $p$ being a Fibonacci-Wieferich prime, we in fact have some extra information. When considering the norm, $N:\Z[\A]\to\Z$, we know that $N(\A^{p^2-1})=1$ since $\A$ is a unit and $p^2-1$ is even. Thus $N(1+k_pp)\equiv1\mod{p^2}$. Computing the norm directly, we obtain
$$\begin{aligned}
N(1+k_pp)&=\left[1+\left(a+b\frac{1+\sqrt5}{2}\right)p\right]\left[1+\left(a+b\frac{1-\sqrt5}{2}\right)p\right]\\
&=1+2ap+bp+a^2p^2+abp^2-b^2p^2.
\end{aligned}$$
So we see that $N(1+k_pp)\equiv1\mod{p^2}$ if and only if $2a+b\equiv0\mod{p}$. As there are $p$ choices for $a$, this results in a $1/p$ chance that $p$ is a Fibonacci-Wieferich prime. 

In addition to this heuristic, there is statistical evidence that there are approximately this many Fibonacci-Wieferich primes. Using the heuristic from Crandall, Dilcher, and Pomerance, both Elsenhans and Jahnel\cite{AJ2010}, and McIntosh and Roettger\cite{RE2007} computed the expected number of ``near misses" in a large interval and compared it to the actual number of near misses. A near miss was defined by first writing $F_{p-\left(\frac{p}{5}\right)}\equiv kp\mod{p^2}$, and if $|k|$ was less than a specified small number, say 100, it was called a near miss. In both papers the expected number and actual number of near misses are nearly the same, and it was not dependent on which interval was used.

However, although all the heuristics suggest infinitely many, no Fibonacci-Wieferich primes have been found in quite a large interval. In his initial paper, Wall gave data for $p<2000$ to support his conjecture, and this has been extended by \cite{PrimeGrid} which shows that there is no Fibonacci-Wieferich prime up to $1.2\times 10^{17}$.

% \begin{acknowledgment}{Acknowledgment.}
% We express our sincere thanks to Thomas Tucker 
% for his help and suggestions, and we are also thankful to the editors and reviewers for their valuable opinions. 
% \end{acknowledgment}

\bibliography{ABCFIB}
\bibliographystyle{amsra}
\begin{comment}

\begin{biog}
\item[George Grell] (george.grell@rochester.edu) was born and raised in Boulder, Colorado. He earned his bachelor's degree in mathematics from the University of Colorado and dabbled in representation theory there. Afterwards, he went to the University of Rochester and is still working on earning a PhD. He is currently working under Thomas Tucker and his primary interests include algebraic and arithmetic number theory.
\begin{affil}
UR Mathematics, 915 Hylan Building, University of Rochester, RC Box 270138  Rochester, NY 14627\\
george.grell@rochester.edu
\end{affil}

\item[Wayne Peng] (jpeng4@ur.rochester.edu) started at National Chung-Cheng University and graduated in 2009. Before becoming a graduate student at the University of Rochester, he was a research assistant guided by Julie T.-Y. Wang at Academia Sinica of Taiwan. Wayne Peng is working towards his PhD at the University of Rochester, and his advisor is Thomas Tucker. The author's research interests are arithmetic number theory, arithmetic geometry and equidistribution.
\begin{affil}
UR Mathematics, 915 Hylan Building, University of Rochester, RC Box 270138  Rochester, NY 14627\\
jpeng4@ur.rochester.edu
\end{affil}
\end{biog}
\end{comment}
\vfill\eject
\end{document}